\documentclass{article}
\usepackage{amssymb,amsmath,amsthm,enumerate,amsfonts}
\usepackage{mathrsfs}
\usepackage{color}
\makeatletter

\@addtoreset{equation}{section}
\makeatother
\def\om{\omega}

\def\N{\mathbb{N}}

\renewcommand{\leq}{\leqslant}
\renewcommand{\geq}{\geqslant}
\providecommand{\om}{\omega}
\providecommand{\eps}{\varepsilon}

\renewcommand{\div}{\operatorname{div}}

\newcommand{\dist}{\operatorname{dist}}

\providecommand{\R}{\mathbb{R}}

\providecommand{\N}{\mathbb{N}}

\renewcommand{\leq}{\leqslant}
\renewcommand{\geq}{\geqslant}

\renewcommand{\div}{\operatorname{div}}
\newcommand{\curl}{\operatorname{curl}}

\DeclareMathOperator\supp{supp}

\newtheorem{Theorem}{Theorem}
\newtheorem{Definition}[Theorem]{Definition}

\newtheorem{Proposition}[Theorem]{Proposition}
\newtheorem{Lemma}[Theorem]{Lemma}
\newtheorem{Remark}[Theorem]{Remark}
\date{\today}
\title{On the motion of  a rigid body in a two-dimensional  ideal flow with vortex sheet initial data}
\begin{document}

\author{Franck Sueur\footnote{CNRS, UMR 7598, Laboratoire Jacques-Louis Lions, F-75005, Paris, France}
\footnote{UPMC Univ Paris 06, UMR 7598, Laboratoire Jacques-Louis Lions, F-75005, Paris, France}
}

\maketitle
\begin{abstract}
A famous result by Delort about the two-dimensional  incompressible Euler equations is the existence of weak solutions when the initial vorticity is a bounded Radon measure with distinguished sign and lies in the Sobolev space $H^{-1}$.
In this paper we are interested in the case where there is a rigid body immersed in the fluid moving under the action of the fluid pressure.
We succeed to prove the existence of solutions \textit{\`a la Delort}  in a particular case with a  mirror-symmetry assumption already considered 
by \cite{LNX06}, where it was assumed in addition  that the rigid body is a fixed obstacle. 
The solutions built here satisfy the energy inequality and  the body acceleration is  bounded.
  \end{abstract}
\section{Introduction}
\subsection{Motion of a body in a two-dimensional ideal flow}
We consider the motion of a body $\mathcal{S}  (t)$  in 
 a planar ideal fluid which therefore occupies at time $t$ the set $\mathcal{F}(t) := \R^2 \setminus \mathcal{S} (t)$. 
  We assume that the body is a closed disk  of radius one and has a uniform density $\rho >0$.
 The equations modelling the dynamics of the system read
\begin{eqnarray}
	\displaystyle \frac{\partial u}{\partial t}+(u\cdot\nabla)u + \nabla p =0 && \text{for} \ x\in \mathcal{F}(t), \label{Euler1}\\
	\div u = 0 && \text{for} \ x\in \mathcal{F}(t) , \label{Euler2}
\\	u\cdot \mathbf{n} = h' (t)  \cdot \mathbf{n} && \text{for}  \  x\in \partial \mathcal{S}  (t),   \label{Euler3}
\\    m  h'' (t) &=&  \int_{ \partial \mathcal{S} (t)} p \mathbf{n}   ds ,  \label{Solide1} 
\\		u |_{t= 0} &=& u_0 ,\label{Eulerci2}
\\  	( h (0) , h' (0) )  &=&   (0,  \ell_0 )  . \label{Solideci}
\end{eqnarray}
	Here  $u=(u_1,u_2)$ and  $p$  denote the velocity and pressure fields,  $m = \rho \pi$  denotes  the mass  of the body while the fluid  is supposed to be  homogeneous of density $1$, to simplified the equations,	
	$\mathbf{n}$ denotes the unit outward normal on $\mathcal{F}(t)$,  $ds$ denotes the integration element on the boundary  $\partial \mathcal{S}(t)$ of the body.
In the equation (\ref{Solide1}), 	$ h (t)$ is the position of the center of mass of the body.

 The equations  (\ref{Euler1})  and (\ref{Euler2}) are the incompressible Euler equations,  the  condition (\ref{Euler3}) means that the boundary is impermeable, the equation  (\ref{Solide1}) is  Newton's balance law for linear  momentum:	the fluid  acts on the body through pressure force.

In the system above we omit  the equation for the rotation of the rigid ball,  which yields that  the angular velocity of the rigid body remains constant when time proceeds, since the angular velocity is not involved in the equations  (\ref{Euler1})-(\ref{Solideci}).

 \subsection{Equations in the body frame}
We start by transferring the previous equations in  the body frame.
We define:
\begin{equation*}
	\left\{
	\begin{array}{l}
	v (t,x)=  u(t,x+h(t)),\\
	q (t,x)=p(t,x+h(t)),\\
	 {\ell}(t)= h' (t) .
	\end{array}\right.
	\label{chgtvar}
\end{equation*}
so that the equations  (\ref{Euler1})-(\ref{Solideci})  become
\begin{eqnarray}
	\displaystyle \frac{\partial v}{\partial t}+\left[(v-\ell)\cdot\nabla\right]v +\nabla q =0 && x\in \mathcal{F}_{0} , \label{Euler11}\\
	\div v = 0 && x\in \mathcal{F}_{0} , \label{Euler12}\\
	v\cdot \mathbf{n} =\ell \cdot \mathbf{n} && x\in \partial \mathcal{S}_0, \label{Euler13}\\
	m\ell'(t)=\int_{\partial \mathcal{S}_0} q \mathbf{n} \ ds  & & \label{Solide11}\\
	v(0,x)= v_0 (x) && x\in \mathcal{F}_{0} ,\label{Euler1ci}\\
	\ell(0)= \ell_0  . \label{Solide1ci}
\end{eqnarray}
where $ \mathcal{S}_{0} $ denotes the closed unit disk, which is the set initially occupied by the solid and  $\mathcal{F}_{0}  := \R^2  \setminus   \mathcal{S}_{0} $ is the one occupied by the fluid.

 \subsection{Vortex sheets data}
 \label{CD0}
Such a problem has been tackled by \cite{ort2} in the case of a smooth initial data with finite kinetic energy,  by  \cite{GS} in the case of Yudovich-like solutions (with bounded vorticities)
and by \cite{shrinking} in the case where the initial vorticity of the fluid has a $L^p_{c}$ vorticity with $p>2$. The index $c$ is used here and in the sequel for ``compactly supported''.
These works provided  the global existence of  solutions.
Actually the result of \cite{ort2} was extended  in \cite{ort1} to the case of a solid of arbitrary form, for which rotation has to be taken into account, and the 
works  \cite{GS} and \cite{shrinking}  deal with an arbitrary form as well. 
Furthermore we will address in a separate paper the case of an initial vorticity in  $L^p_{c}$  with $p>1$,  in order to achieve the 
investigation of solutions ``\`a la DiPerna-Majda", referring here to the seminal work \cite{DiPernaMajda} in the case of a fluid alone.

It is therefore natural to try to extend these existence results  to the case, more singular,  of vortex sheet initial data.
In the case of a fluid alone, without any moving body, vortex sheet motion is a classical topic in fluid dynamics.
Several approaches have been tried. Here we will follow the approach initiated by J.-M. Delort who proved global-in-time existence of weak solutions for the incompressible Euler equations when the initial vorticity is a compactly supported, bounded Radon measure with distinguished sign in the Sobolev space $H^{-1}$.
The pressure smoothness in Delort's result is very bad so that it could be a-priori argued that 
the extension to the case of an immersed body should be challenging since the motion of the solid is determined by the pressure forces exerted by the fluid on the solid boundary. 
However the problem \eqref{Euler11}--\eqref{Solide1ci} admits a global weak formulation  where the pressure disappears.
The drawback is that test functions involved in this weak formulation do not vanish on the interface between the solid and the fluid, 
an unusual fact in Delort's approach, where the solution rather satisfies a weak formulation of the equations which involves some test functions compactly supported in the fluid domain (which is open) and the boundary condition is prescribed in a trace sense.

Yet in the  paper \cite{LNX01}, the authors deal with the case of an initial vorticity compactly supported, bounded Radon measure with distinguished sign in $H^{-1}$ in the upper half-plane, superimposed on its odd reflection in the lower half-plane. The corresponding  initial velocity is then mirror symmetric with respect to the horizontal axis.
 In the course of proving the existence of solutions to this problem, they are led to introduce another notion of weak solution that they called boundary-coupled weak solution, which relies on a weak vorticity formulation which involves some test functions that vanish on the boundary, but not their derivatives.
They have extended their analysis to the case of a fluid occupying the exterior of a symmetric fixed body in 
\cite{LNX06}.

 \subsection{Mirror symmetry}
\label{notat}

In this paper, we assume that the initial velocities $\ell_0$ and $v_0$ are mirror symmetric with respect to the horizontal axis given by the  equation $x_2 = 0$. Our setting here can therefore be seen as an extension of the one in \cite{LNX06} from the case of a fixed obstacle to the case of a moving body.

For the body velocity  $\ell_0 \in \R^2$ the mirror symmetry  entails that $\ell_0$ is of the form $\ell_0 =(\ell_{0,1},0)$.
Let us now turn our attention to the fluid velocity. 
Let
$$\mathcal{F}_{0,\pm} := \{ x\in \mathcal{F}_{0}/  \ \pm x_{2} > 0 \}
\text{ and } \Gamma_\pm := \partial \mathcal{F}_{0,\pm} .$$
If $x=(x_1 ,x_2 ) \in \mathcal{F}_{0,\pm}$ then we denote $\tilde{x} := (x_1 , -x_2 ) \in \mathcal{F}_{0,\mp}$.
To avoid any confusion let us say here that for a smooth vector field $u= (u_1 ,u_2 )$, the  mirror symmetry assumption means that  for any $x  \in \mathcal{F}_{0,\pm}$,  $(u_1 ,u_2 ) (\tilde{x} ) = (u_1 , - u_2 ) (x)$.

This assumption has two important consequences. First  the vorticity $\omega := \curl v := \partial_1 v_2 -   \partial_2 v_1$ is odd with respect to the variable $x_2$  and therefore its integral over the fluid domain $\mathcal{F}_{0}$ vanishes. The other one is that  the circulation of the initial velocity around the body vanishes.
In such a case, it is very natural to consider finite energy velocity. 

We will explain why in the next section by considering the link between velocity and vorticity.

Let us also mention that the analysis performed here can be adapted to the case of a body occupying  a smooth, bounded, simply connected closed set, which is   symmetric with respect to the horizontal coordinate axis, and which is not allowed to rotate (for instance because of   the action of an exterior torque on the body preventing any rotation, or because the angular mass is infinite).
 \subsection{A velocity decomposition}
We denote by 
%[ ]
\begin{eqnarray*}
G(x,y) := \frac{1}{2\pi} \ln \frac{ | x-y| }{ | x-y^{*}| \,  | y| } , \text{ where } y^{*} := \frac{y}{| y^{*} |^{2}}  , 
\end{eqnarray*}
 the Green's function of $\mathcal{F}_{0}$ with Dirichlet boundary condition. We also introduce the functions
%[ ]
\begin{eqnarray}
\label{BSexplicit}
H(x) :=  \frac{(x-y)^\perp}{2\pi | x - y |^{2}}   \text{ and } K(x,y) :=H(x-y) - H(x-y^{*} ), 
\end{eqnarray}
 with the notation $x^\perp := ( -x_2 , x_1 )$ when $x=(x_1,x_2)$, 
 which are  the kernels of the Biot-Savart operators respectively in the full plane and in $\mathcal{F}_{0}$.
 More precisely we define  the operator  $K[\om]$ as acting on
 $\om\in C^\infty_c (\mathcal{F}_{0})$  through the formula 
\begin{equation*}
  K[\om](x)=\int_{\mathcal{F}_{0}}K(x,y) \om(y) d y.
\end{equation*}
We will extend this definition to bounded Radon measures in the sequel but let us consider here the smooth case first to clarify the presentation. 
We also define the hydrodynamic Biot-Savart  operator  $K_\mathcal{H} [\om]$ by
\begin{equation*}
  K_\mathcal{H}[\om](x)=\int_{\mathcal{F}_{0}} K_\mathcal{H}(x,y) \om(y) d y \text{ with }  K_\mathcal{H}(x,y) := K(x,y) + H(x) .
\end{equation*}
One easily verifies that 
\begin{equation}
\label{infiny}
 \lim_{| x|+| y| \rightarrow +\infty } K_\mathcal{H}(x,y)  =0
 \end{equation}
  and that 
$H$ and $K_\mathcal{H}[\om]$ satisfy
\begin{eqnarray}
\label{PbH}
\div H = 0 , \, \curl H = 0 \text{ in }  \mathcal{F}_{0}  , \,  H \cdot \mathbf{n} = 0 \text{ in } \partial  \mathcal{S}_{0} , \,  \int_{\partial \mathcal{S}_0} H \cdot   \mathbf{n}^\perp ds =  - 1 , \, \lim_{| x| \rightarrow +\infty } H  =  0 ,
\\ \label{PbK}
\div K_\mathcal{H} [\om] = 0 , \, \curl K_\mathcal{H} [\om] = \om \text{ in }  \mathcal{F}_{0}  , \, K_\mathcal{H} [\om]  \cdot \mathbf{n} = 0 \text{ in } \partial  \mathcal{S}_{0} , \,  \int_{\partial \mathcal{S}_0} K_\mathcal{H} [\om]  \cdot   \mathbf{n}^\perp ds = 0 , \, \lim_{| x| \rightarrow +\infty } K_\mathcal{H} [\om] =0   .
\end{eqnarray}

Let us also define the Kirchhoff potentials
\begin{eqnarray*}
\Phi_i (x) := - \frac{x_i  }{ | x|^{2}}  ,
\end{eqnarray*}
which satisfies
\begin{equation}
\label{t1.3sec}
	-\Delta \Phi_i = 0  \text{ for }  \ x\in \mathcal{F}_{0}   , \quad  \Phi_i \rightarrow 0  \text{ for}  \ x \rightarrow  \infty,  \quad	\frac{\partial \Phi_i}{\partial \mathbf{n}}= \mathbf{n}_i     \text{ for }  \  x\in \partial \mathcal{S}_{0}   ,
\end{equation}
for $i=1,2$, where $\mathbf{n}_1 $ and $\mathbf{n}_2$  are the components of the normal vector $\mathbf{n}$.
Let us also observe $\nabla \Phi_i $  is in $C^\infty (\overline{\mathcal{F}_{0}}) \cap L^2 (\mathcal{F}_{0})$,   and that the derivatives of higher orders of $\nabla \Phi_i $ are also in $L^2 (\mathcal{F}_{0})$.

Then we have the following decomposition result :
\begin{Lemma}
\label{decopsmooth}
Let  $\om\in C^\infty_c (\mathcal{F}_{0})$, $\ell:=( \ell_1 , \ell_2) \in \R^2$ and $\gamma \in \R$. Then
there exists one only smooth divergence free vector field $u$ such that 
$u\cdot \mathbf{n} = \ell\cdot  \mathbf{n} $ on $\partial \mathcal{S}_0$, $\int_{\partial \mathcal{S}_0} u \cdot   \mathbf{n}^\perp ds = \gamma $, $\curl u = \om $ in $\mathcal{F}_{0}$ and such that $u$ vanishes at infinity. Moreover $u =  K[\om] +  \ell_1  \nabla  \Phi_1 + \ell_2  \nabla  \Phi_2 + (\alpha -\gamma ) H$, where $\alpha :=  \int_{\mathcal{F}_{0}} \om dx$.
\end{Lemma}
\begin{proof}
Combining \eqref{PbH}, \eqref{PbK} and \eqref{t1.3sec} we get  the existence part. 
Regarding the uniqueness, it is sufficient to apply \cite[Lemma 2.14]{Kikuchi83}.

\end{proof}
Now our point is that considering some  mirror symmetric velocities $u$ and $\ell$, assuming again that $u$ is smooth with $ \om := \curl u$ in $C^\infty_c (\mathcal{F}_{0})$, one has $\int_{\partial \mathcal{S}_0} u \cdot   \mathbf{n}^\perp ds = 0$ and $ \int_{\mathcal{F}_{0}} \om dx = 0$, so that, according to the previous lemma,   $u =  K[\om] +  \ell_1  \nabla  \Phi_1 $. One then easily infers from the definitions above that $u  \in L^2 (\mathcal{F}_{0} )$.
The kinetic energy  $m \ell^2 + \int_{\mathcal{F}_{0}} u^2 dx$ of the system ``fluid+body'' is therefore finite.
Let us also stress that $u$ can also be written as  $u =  K_\mathcal{H} [\om] +  \ell_1  \nabla  \Phi_1 $. 
Here the advantage of using $K_\mathcal{H} [\om]$ rather than  $K [\om]$ is that we will make use of \eqref{infiny}, which is not satisfied by $K(x,y)$.
 \subsection{Cauchy data}
 \label{CD}
Let us now define properly the Cauchy data we are going to consider in this paper.
For a subset $X$ of $\R^2$ we will use the notation $\mathcal{BM} (X)$ for the set of the bounded measures  over $X$, $\mathcal{BM}_+ (X)$ for the set of the positive measures  over $X$, 
$\mathcal{BM}_c (X)$ the subspace of the measures of $\mathcal{BM} (X)$  which are  compactly supported in  $X$ and, following the terminology of   \cite{LNX06}, we will say that a  $\om \in \mathcal{BM} (\mathcal{F}_{0})$ is nonnegative mirror symmetric (NMS) if it is odd with respect to the horizontal axis and if it is nonnegative in the upper half-plane.
This means that for any $\phi \in C_c (\mathcal{F}_{0} ; \R)$, 
\begin{eqnarray}
\label{distinguo}
\int_{\mathcal{F}_{0}} \phi (x) d\omega (x) = - \int_{\mathcal{F}_{0}} \phi (\tilde{x}) d\omega (x) ,
\end{eqnarray}
 with the notation of Section \ref{notat}.

We now extend the operator $K[\cdot]$ to any $\om \in \mathcal{BM} (\mathcal{F}_{0})$ by defining  $K[\om] \in \mathcal{D}' (\mathcal{F}_{0})$ through the formula
\begin{equation*}
  \forall f \in C^\infty_c (\mathcal{F}_{0}) , \quad <K[\om] , f >= \int_{\mathcal{F}_{0}}G*\curl f \ d \om.
\end{equation*}
Let $\ell_{0,1} \in \R$ and  $\ell_0 =(\ell_{0,1},0)$.
Let  $\omega_{0,+} \in \mathcal{BM}_{c,+} (\mathcal{F}_{0,+})$ and $\omega_{0,-}$  the corresponding measure in $\mathcal{F}_{0,-}$ obtained by odd reflection.
We then denote $\omega_0 := \omega_{0,+} + \omega_{0,-}$  which is in $\mathcal{BM} (\mathcal{F}_{0})$ and is NMS.
We define accordingly the initial fluid velocity by 
$v_0 :=   K[\om_0] +  \ell_{0,1}  \nabla  \Phi_1 $.

 \subsection{Weak formulation}
Let us now give a global weak formulation of the problem by considering -for solution and for test functions- a velocity field on the whole plane, with the constraint to be constant on $ \mathcal{S}_{0} $.
We introduce the following space
\begin{equation*}
	\mathcal{H} =\left\{\Psi\in L^2(\mathbb{R}^2) ;  \quad  \div \Psi =0 \quad \text{in} \ \mathbb{R}^2,\  \nabla \Psi =0 \quad \text{in} \ \mathcal{S}_0 \right\},
\end{equation*}
which is a Hilbert space endowed with the scalar product
\begin{equation}
\label{scalarp}
(\overline{u},\overline{v})_\rho := \int_{\mathbb{R}^2} (\rho \chi_{\mathcal{S}_{0} } + \chi_{\mathcal{F}_{0} })
\overline{u} \cdot \overline{v} = m \, \ell_u \cdot \ell_v + \int_{\mathcal{F}_{0} } u \cdot v  ,
\end{equation}
where the notation $\chi_{A} $ stands for the characteristic funtion of the set $A$, $\ell_u \in \mathbb{R}^2$ and $ u  \in  L^2(\mathcal{F}_{0})$ denote respectively  the restrictions of $\overline{u}$ to $\mathcal{S}_0$ and $\mathcal{F}_{0}$. Let us stress here that because, by definition of  $\mathcal{H}$, $\overline{u}$   is assumed to satisfy the divergence free condition in the whole plane, the normal component of these restrictions have to match on the boundary $\partial \mathcal{S}_0$.
We will denote $\|  \cdot \|_\rho$ the norm associated to $(\cdot ,\cdot )_\rho$.
Let us also introduce $\mathcal{H}_{T}$ the set of the test functions $\Psi$ in $C^1([0,T]; \mathcal{H} )$ with its restriction $\Psi |_{[0,T] \times  \overline{\mathcal{F}_{0}}}$ to the closure of the fluid domain in $C^1_c ( [0,T] \times \overline{\mathcal{F}_{0}})$.
\begin{Definition}[Weak Solution] 
\label{DefWS}
Let be given  $ \overline{v}_0 \in \mathcal{H}$ and $T>0$.
We say that  $ \overline{v} \in C ([0,T]; \mathcal{H}-w)$ is a weak solution of \eqref{Euler11}--\eqref{Solide1ci} in $[0,T]$ if for any test function $\Psi \in \mathcal{H}_{T} $,
\begin{equation}
	(  \Psi(T,\cdot ) , \overline{v}(T,\cdot) )_\rho
	- (  \Psi (0,\cdot ) , \overline{v}_0 )_\rho
	= \int_0^T (\frac{\partial \Psi}{\partial t},\overline{v} )_\rho \ dt
	+\int_0^T \int_{\mathcal{F}_0} v\cdot\left(\left(v-\ell_v \right)\cdot\nabla\right)\Psi \ dx\ dt
	\label{dws2}
\end{equation}
\end{Definition}
Definition \ref{DefWS} is legitimate since a classical solution of \eqref{Euler11}--\eqref{Solide1ci} in $[0,T]$ is also a weak solution. This follows easily from an integration by parts in space which provides
\begin{eqnarray*}
 (\partial_{t} \overline{v} , \Psi  )_\rho 
	 = \int_{\mathcal{F}_0} v \cdot\left(\left(v-\ell_v \right)\cdot\nabla\right)\Psi \ dx ,
\end{eqnarray*}
and then from an  integration by parts in time.

 \subsection{Main result}
Our main result is the following.
\begin{Theorem}
 \label{DelortBody}
 Let be given a Cauchy data  $ \overline{v}_0 \in \mathcal{H}$ as described in Section \ref{CD}. Let $T>0$.
 Then there exists a weak solution  of \eqref{Euler11}--\eqref{Solide1ci} in $[0,T]$. 
 In addition this solution preserves the mirror-symmetry and satisfies the energy inequality: for any $t \in [0,T]$,
 $\| \overline{v} (t,\cdot) \|_\rho \leq \| \overline{v}_0  \|_\rho .$
 Moreover the acceleration $\ell'$ of the body is bounded in $[0,T]$.
 \end{Theorem}
Let us slightly precise the last assertion. 
Actually the proof will provide a bound of  $\| \ell'  \|_{L^{\infty} (0,T) } $ which only depends on the body mass $m$ and on the initial energy $ \|   \overline{v}_0 \|_\rho$, but not on $T$.

Let us also stress that it is straightforward, by an energy estimate, to prove that the weak solution above enjoys a weak-strong uniqueness property. 
Then, applying Th. 1 of \cite{laure}, it follows that uniqueness holds for a $G_\delta$ dense subset of $\mathcal{H}$ endowed with its weak topology.

The rest of the paper is devoted to the proof of Theorem \ref{DelortBody}.

 \section{Proof}

A general strategy for obtaining a weak solution is to smooth out the initial data so that one gets a sequence of  initial data which launch some classical solutions, and then to pass 
to the limit with respect to the regularization parameter in the weak formulation of the equations.  

\subsection{A regularized sequence}
\label{wc}
Let $( \eta_n)_n$ be a sequence of even mollifiers. We therefore consider the sequence of regularized initial vorticities $( \omega_{0}^n )_n$ given by $ \omega_{0}^n := ( \omega_{0}) * \eta_n $
 and some corresponding initial velocities $(\overline{v}^{n}_{0})_n$ in $\mathcal{H}$ with
\begin{eqnarray*}
\overline{v}^{n}_{0} := \ell_0    \text{ in } \mathcal{S}_{0} \text{ and } \overline{v}^{n}_{0} := v^{n}_{0} := K[ \omega_{0}^n ] + \ell_{0,1} \nabla     \Phi_1  \text{ in } \mathcal{F}_{0}.
\end{eqnarray*}
Then   the $( \omega_{0}^n )_n$ are smooth, compactly supported in $\mathcal{F}_0$ (at least for $n$ large enough), NMS and bounded in $ L^1(\mathcal{F}_{0})$, and $(\overline{v}^{n}_{0})_n$ converges weakly in $\mathcal{H}$ to $ v_{0}$.

Let $( \overline{v}^n )_{n}$  in $C ([0,T]; \mathcal{H})$ be the classical solutions of \eqref{Euler11}--\eqref{Solide1ci} in $[0,T]$ respectively associated to the sequence  $(\overline{v}^n_0 )_{n}$ of initial data (cf. \cite{ort2}).
According to Lemma \ref{decopsmooth} the restriction $v^n$ of $\overline{v}^n$ to $\mathcal{F}_{0}$ splits into 
\begin{equation}
\label{vdecompo}
v^n = u^{n} +   \nabla \Phi^{n} \text{ where } u^{n} := K [\omega^n ]    \text{ and } \Phi^{n} :=  \ell^n_1 \Phi_1   .
\end{equation}
Observe in particular that from now on we denote $\ell^{n}$ for $\ell_{v^n} $.

Moreover these solutions preserve, for any $t$  in $[0,T]$, the mirror symmetry (this follows from the uniqueness of the Cauchy problem for classical solutions), the  kinetic energy: 
\begin{eqnarray}
\label{conservvelo}
\|  \overline{v}^n (t,\cdot) \|_{\rho} = \|   \overline{v}^n_0 \|_{\rho} ,
\end{eqnarray}
and  the $L^1$ norm of the vorticity on the upper and lower half-planes:
\begin{eqnarray}
\label{conservvorty}
\|  \omega^n  (t,\cdot) \|_{L^{1} (\mathcal{F}_{0,\pm})} = \|  \omega^n_0 \|_{L^{1} (\mathcal{F}_{0,\pm})} .
\end{eqnarray}
This last property can be obtained from the vorticity  equation:
\begin{eqnarray}
\label{eqvorty}
\partial_t \omega^n  + (v^n - \ell^{n} ) \cdot \nabla \omega^n  = 0 .
\end{eqnarray}
As already said before a classical solution is a fortiori a weak solution, 
thus for any test function $\Psi$ in $ \mathcal{H}_{T} $, 
\begin{equation}
	(  \Psi(T,\cdot ) , \overline{v}^n (T,\cdot) )_\rho
	- (  \Psi (0,\cdot ) , \overline{v}^n_0 )_\rho
	= \int_0^T (\frac{\partial \Psi}{\partial t},\overline{v}^n )_\rho \ dt
	+\int_0^T \int_{\mathcal{F}_0} v^n \cdot\left(\left(v^n-\ell^{n} \right)\cdot\nabla\right)\Psi \ dx\ dt .
	\label{dws29}
\end{equation}
Using the bounds 
\eqref{conservvelo} and \eqref{conservvorty}, we obtain that there exists a subsequence $( \overline{v}^{n_{k}} )_{k}$ of $(\overline{v}^n  )$ which converges to $ \overline{v}$ in $L^{\infty} ((0,T); \mathcal{H} )$ weak* and such that $( \omega^{n_{k}} )_{k}$ converges to $ \omega$ weak* in $L^{\infty} ((0,T);\mathcal{BM} (\overline{\mathcal{F}_0} ))$. In particular we have 
that $( \ell^{n_{k}} )_{k}$  converges to $\ell$ in $L^{\infty} (0,T)$  weak* and  that $( v^{n_{k}} )_{k}$  converges to $v$ weak* in $L^{\infty} ((0,T); L^{2}  (\mathcal{F}_{0}) )$ weak*, where $\ell$ and $v$  denote respectively the restrictions to $\mathcal{S}_{0}$ and $\mathcal{F}_{0}$ of $ \overline{v}$, and are  mirror symmetric, so that the vector  $\ell$ is of the form $(\ell_1 ,0)$. We also have that  $\om$  is NMS. 
In particular $\om$ has a vanishing total mass, that is  $ \omega (t, \cdot) (\mathcal{F}_{0})=0$ for almost every $t\in (0,T)$.
One should wonder whether or not the oddness holds in $\overline{\mathcal{F}_{0}}$ as well, that is if  \eqref{distinguo} also holds true for  $\phi \in C_c (\overline{\mathcal{F}_{0}})$. Actually we will see later that, for almost every time,  the measure  $\om$ of the boundary  vanishes, what implies a positive answer.

Our goal now is  to prove that the limit obtained  satisfies the weak formulation \eqref{dws2}.
Unfortunately, the weak convergences above are far from being sufficient to pass to the limit. 
We will first improve these convergences with respect to the time variable. 
More precisely in the next section we will give an estimate of the body acceleration which will allow to obtain strong convergence in $C( [0,T ]) $ of a subsequence of the solid velocities.
Then we will give  an estimate of the time derivative of the vorticity  which will allow to obtain  strong convergence  in $C( [0,T ] ; \mathcal{BM}(\mathcal{F}_{0} ) - w^* )$ of a subsequence of the vorticities.
Finally  we will pass to the limit thanks to an argument of no-concentration of the vorticity, up to the boundary. 

 \subsection{Estimate of the body acceleration}

The goal of this section is to prove the following.
\begin{Lemma}
\label{Addedm}
The sequence $ (( \ell^n ) ' )_{n}$ is bounded in $L^{\infty}(0,T)$.
\end{Lemma}
\begin{proof}
Let $\ell$ be in $\R^{2}$. Then we define  $\Psi$ in $ \mathcal{H} $ by setting
$\Psi = \ell$ in $\mathcal{S}_{0}$ and $\Psi  = \nabla ( \ell_{1}\Phi_{1} + \ell_{2}\Phi_{2} )$ in $\mathcal{F}_{0}$.
Therefore, $ \overline{v}^{n}$ being a classical solution of the system \eqref{Euler11}--\eqref{Solide1ci}, one has 
\begin{equation}
\label{Tempobody}
	 (\partial_{t} \overline{v}^n , \Psi  )_\rho 
	 = \int_{\mathcal{F}_0} v^n \cdot\left(\left(v^n-\ell^n \right)\cdot\nabla\right)\Psi \ dx .
	%= \check{T}^{n,\eps}_{1} +  \tilde{T}^{n,\eps}_{1} + T^{n}_{2} + T^{n}_{3} .
\end{equation}

By using the definition of the scalar product in 
\eqref{scalarp}, 
 \eqref{t1.3sec} and the boundary condition  \eqref{Euler13}
 we obtain 
\begin{eqnarray*}
	 (\partial_{t} \overline{v}^n , \Psi  )_\rho 
	 &=&\ell^{T} \mathcal{M} ( \ell^n ) ' ,
\end{eqnarray*}
with
$$  \mathcal{M} := m Id_{2} + (  \int_{\mathcal{F}_0}   \nabla \Phi_{i} \cdot   \nabla  \Phi_{j} dx )_{i,j} ,$$
which is a $2\times 2$ positive definite symmetric matrix that stands for the added mass of the body  which, loosely speaking,  measures how much the  surrounding fluid resists the acceleration as the body moves through it.

Now we use that $\nabla \Psi $ is in $L^{2} (\mathcal{F}_0 ) \cap L^{\infty} (\mathcal{F}_0 )$ and \eqref{conservvelo} to get that the right hand side of \eqref{Tempobody} is bounded 
uniformly in $n$. 
Therefore $ ( \ell^n ) ' $ is bounded in $L^{\infty}(0,T)$.
\end{proof}
In particular we deduce from this, \eqref{conservvelo} and Ascoli's theorem  that there exists a subsequence, that we still denote $( \ell^{n_{k}} )_{k}$, which   converges  strongly to $\ell$ in $C([0,T ])$.
Moreover by weak compactness, we also have  that 
 $( (\ell^{n_{k}} )' )_{k}$  converges to $\ell'$ in $L^{\infty} (0,T)$  weak*.

 \subsection{A decomposition of the nonlinearity}
The main difficult term  to pass to the limit into \eqref{dws29} is the third one because of its nonlinear feature.
 We first use \eqref{vdecompo} to obtain for any test function $\Psi$ in $\mathcal{H}_{T} $, 
\begin{eqnarray*}
\int_0^T \int_{\mathcal{F}_0} v^n \cdot\left(\left(v^n-\ell^n \right)\cdot\nabla\right)\Psi \ dx\ dt &=&
  T^n_{1} + T^n_{2} + T^n_{3}  \text{ where } 
\\  T^n_{1}&:=&  \int_0^T \int_{\mathcal{F}_0} u^n \cdot \left(\left(u^n \right)\cdot\nabla\right)\Psi \ dx \ dt , 
\\   T^n_{2}&:=& \int_0^T  \int_{\mathcal{F}_0} u^n \cdot \left(\left( \nabla  \Phi^{n}  - \ell^n \right)\cdot\nabla\right)\Psi \ dx \ dt ,
\\  T^n_{3}&:=& \int_0^T  \int_{\mathcal{F}_0}  \nabla  \Phi^{n} \cdot \left(\left( u^n + \nabla  \Phi^{n}  - \ell^n \right)\cdot\nabla\right)\Psi \ dx \ dt .
\end{eqnarray*}
From what precedes we infer that 
$(  T^{n_{k}}_{2} )_k$ and $(  T^{n_{k}}_{3} )_k$ converge respectively to $  T_{2}$ and  $  T_{3}$, where
\begin{eqnarray*}
 T_{2}:= \int_0^T  \int_{\mathcal{F}_0} u \cdot \left(\left( \nabla  \Phi  - \ell \right)\cdot\nabla\right)\Psi \ dx  \ dt,
\quad  T_{3}:= \int_0^T  \int_{\mathcal{F}_0}  \nabla  \Phi \cdot \left(\left( u + \nabla  \Phi  - \ell \right)\cdot\nabla\right)\Psi \ dx \ dt ,
\end{eqnarray*}
where $ \Phi :=  \ell_{1} \Phi_{1} $.

	The term $T^n_{1}$ is more complicated.  
	We would like to use vorticity to deal with this term, as in Delort's method where ruling out vorticity concentrations (formation of Dirac masses) allows to deal with the nonlinearity.
	However there is a difference here: the test function $\Psi$ involved  in the term $T^n_{1}$ is not vanishing in general in the neighborhood of the boundary $\partial  \mathcal{F}_0$. 
	We will use several arguments to fill this gap. 
	In the next section we point out the role played by the  normal trace of  test functions. 
	
 \subsection{Introduction of the vorticity in the nonlinearity}
Let us  start with the following lemma.
\begin{Lemma}
\label{Never00}
Let $ \omega $ be smooth compactly supported  in $\mathcal{F}_{0}$ such that   $u := K[ \omega ]  $ is   in $ L^2 (\mathcal{F}_0 )$.
 Then, for  $ \Psi \in C^1_c ( \overline{\mathcal{F}_{0}})$ divergence free,
\begin{eqnarray}
\label{Never00Eq2}
\int_{\mathcal{F}_{0}} u \cdot ( u \cdot \nabla \Psi ) dx 
= 
- \frac{1}{2}  \int_{ \partial \mathcal{S}_0} | u  \cdot  \mathbf{n}^{\perp} |^{2} \Psi  \cdot \mathbf{n} 
+ 
\int_{\mathcal{F}_{0}}  \omega u\cdot  \Psi^{\perp}  .
 \end{eqnarray}
Assume in addition that  $ \omega $  is NMS, then, also for  $ \Psi \in C^1_c ( \overline{\mathcal{F}_{0}})$ divergence free,
\begin{eqnarray}
\label{Never00Eq}
\int_{\mathcal{F}_{0}^\pm} u \cdot ( u \cdot \nabla \Psi ) dx 
= 
- \frac{1}{2}  \int_{\Gamma_\pm} | u  \cdot  \mathbf{n}^{\perp} |^{2} \Psi  \cdot \mathbf{n} 
+ 
\int_{\mathcal{F}_{0}^\pm}  \omega u\cdot  \Psi^{\perp}  .
 \end{eqnarray}
\end{Lemma}
\begin{proof}
Let us focus on the proof of \eqref{Never00Eq}; the proof of \eqref{Never00Eq2} being similar.
First we observe that $u$ is smooth, divergence free,  in $ L^2 (\mathcal{F}_0 )$ and is tangent to $\Gamma_\pm$ (since   $ \omega $  is NMS).
Now, using that $u$ and $\Psi$ are divergence free, we obtain 
\begin{eqnarray}
\label{craz}
u \cdot ( u \cdot \nabla \Psi ) = u^\perp \cdot \nabla (\Psi^\perp \cdot u ) +  \Psi \cdot \nabla (  \frac{1}{2} | u |^{2} ) .
\end{eqnarray}
Therefore integrating by parts, using that $u$ is tangent to $\Gamma_\pm$, that $\div u^\perp = -  \omega $ and that $ \Psi $  is  divergence free, we get the desired result.

\end{proof}
Let us  first recall what happens when $\Psi$ is in $C^1_c (\mathcal{F}_0 )$. This will already provide some useful informations in the next section. 
\begin{Lemma}
\label{Never0}
Let $ \omega $ in $\mathcal{BM} (\mathcal{F}_{0})$, diffuse (that is $\omega (\{ x\} ) =0$ for any $x \in \mathcal{F}_{0}$), with vanishing total mass, such that  $u := K[ \omega  ]  \in L^2 (\mathcal{F}_0 )$.
Let $ \Psi \in C^1_c (\mathcal{F}_0 )$ divergence free.
Then 
\begin{eqnarray}
\label{nonli}
\int_{\mathcal{F}_{0}} u \cdot ( u \cdot \nabla \Psi ) dx 
= - \frac{1}{2} \iint_{\mathcal{F}_{0} \times \mathcal{F}_{0} } H_{\Psi^{\perp}}(x,y) 
 \ d \omega (x) d \omega (y)  ,
\end{eqnarray}
where
\begin{eqnarray*}
H_f (x,y) := f(x)\cdot K_\mathcal{H} (x,y) + f(y) \cdot K_\mathcal{H} (y,x) .
\end{eqnarray*}
\end{Lemma}

When  $ \omega $ is smooth, the previous lemma follows from Lemma \ref{Never00}: it suffices to plug the definition of the Biot-Savart operator in the second term of the right hand side of \eqref{Never00Eq2} and to symmetrize.
The gain of this symmetrization is that the auxiliary function $H_f (x,y) $ is bounded, whereas the Biot-Savart kernels $K(x,y)$ and $ K_\mathcal{H} (x,y)$ are not.
More precisely it also follows from the analysis in  \cite{Schochet95} that:
\begin{Proposition}\label{propestK0}
There exists a constant $M_2$ depending only on $\mathcal{F}_{0}$ such that
\begin{equation}\label{estK}
| H_f (x,y)  |\leq M_2\|f\|_{W^{1,\infty} ( \mathcal{F}_{0})}
\quad\forall x,y\in\mathcal{F}_{0},\ x\neq y.  
\end{equation}
for any $f\in C^{1}_c ({ \mathcal{F}_{0}};\R^2)$.
\end{Proposition}
Proposition \ref{propestK0} is also true if one substitutes $K(x,y)$ to $ K_\mathcal{H} (x,y)$ in the definition of $H_f$ above.
However the choice of  $ K_\mathcal{H} (x,y)$ seems better since it implies the  extra property that for  any $f\in C^{1}_c ({ \mathcal{F}_{0}};\R^2)$, $H_f$ is tending to $0$ at infinity, thanks to \eqref{infiny}.

Using this, one infers that Lemma \ref{Never00} also holds true for any  diffuse measure by a regularization process. Let us refer again here to \cite{Schochet95} for more details, or to the sequel of this paper where we will slightly extend this.

 \subsection{Temporal estimate of the fluid}
 We have the following.
\begin{Lemma}
\label{tensor}
There exists a subsequence  $(\overline{v}^{n_{k}})_{k}$ of $(\overline{v}^{n})_{n}$ which converges to $\overline{v}$ in $C( [0,T ] ; \mathcal{H} - w )$, and such that $(\omega^{n_{k}})_{k}$ of $(\omega^n)_{n }$  converges to $\omega := \curl {v}$ in $C( [0,T ] ; \mathcal{BM}(\mathcal{F}_{0}) - w^* )$. 
\end{Lemma}
\begin{proof}

Let us consider a divergence free vector field  $\Psi$ in  $C^\infty_c ( {\mathcal{F}_{0}})$, so that
\begin{equation*}
	  \int_{\mathcal{F}_0} \Psi \cdot \partial_{t } {v}^n dx =  ( \Psi ,\partial_{t }\overline{v}^n )_\rho 	 =  {T}^n_{1} + T^{n}_{2} + T^{n}_{3} ,
\end{equation*}
where, thanks to  Lemma \ref{Never0},
\begin{equation*}
{T}^n_{1} := - \frac{1}{2} \iint_{\mathcal{F}_{0} \times \mathcal{F}_{0} } H_{\Psi^{\perp}}(x,y)   
 \ \omega^n (x) \omega^n (y)  dx dy .
\end{equation*}
We can infer from Proposition
\ref{propestK0} and \eqref{conservvelo} that 
\begin{equation*}
|    \int_{\mathcal{F}_0} \Psi  \cdot \partial_{t } {v}^n dx  |   \leqslant C \| \Psi \|_{H^{1} \cap W^{1,\infty}(\mathcal{F}_{0})} .
 \end{equation*}
Moreover using that for any $\phi \in C^\infty_c (  {\mathcal{F}_{0}})$ then  $\Psi = \nabla^\perp \phi $ is  a divergence free vector field   in  $C^\infty_c ( {\mathcal{F}_{0}})$, we get 
\begin{equation*}
|    \int_{\mathcal{F}_0} \phi  \cdot \partial_{t } \omega^n dx  | =  |    \int_{\mathcal{F}_0} \Psi  \cdot \partial_{t } {v}^n dx  |  \leqslant C \| \Psi \|_{H^{1} \cap W^{1,\infty}(\mathcal{F}_{0})} \leqslant C \| \phi \|_{H^{2} \cap W^{2,\infty}(\mathcal{F}_{0})} .
 \end{equation*}
%
%Therefore  $(\partial_t {v}^n )_n$ is is bounded in $L^{\infty}((0,T);  H^{-M} (\mathcal{F}_{0}))$ for $M>2$.
It is therefore sufficient to  use the Sobolev embedding theorem and  the following version of the Aubin-Lions lemma with $M>2$ and with 
\begin{enumerate}
\item $X= L^2 (\mathcal{F}_{0 } )$, $Y= H^{M}_{0}(\mathcal{F}_{0 } ) $,  the completion of $C^\infty_c (  {\mathcal{F}_{0}})$ in the Sobolev space 
$H^{M}(\mathcal{F}_{0 } ) $,  and $f_n = {v}^n$; and with 
\item $X= C_{0} (\mathcal{F}_{0 } )$ and $Y =H^{M+1}_{0}(\mathcal{F}_{0 } )$ and $f_n = \omega^n$.
\end{enumerate}
\begin{Lemma}
\label{weakc}
Let $X$ and $Y$ be two separable Banach spaces such that $Y$ is dense in $X$. Assume that $(f_n )_n $ is a bounded sequence in  $L^{\infty}((0,T); X')$ such that  $(\partial_t f_n )_n $ is  bounded in $L^{\infty}((0,T); Y')$. Then $(f_n )_n $ is relatively compact in  $C( [0,T ]; X' - w^* )$.
\end{Lemma}
The proof of Lemma \ref{weakc} is given in appendix for sake of completeness.
\end{proof}
A first consequence of the previous result is that we can pass to the limit the left hand side of \eqref{dws29}:
for any test function $\Psi$ in $ \mathcal{H}_{T}$, as $k \rightarrow +\infty$, 
\begin{equation*}
	(  \Psi(T,\cdot ) , \overline{v}^{n_{k}} (T,\cdot) )_\rho
	- (  \Psi (0,\cdot ) , \overline{v}^{n_{k}}_0 )_\rho
	\rightarrow (  \Psi(T,\cdot ) , \overline{v} (T,\cdot) )_\rho
	- (  \Psi (0,\cdot ) , \overline{v}_0 )_\rho .
	\end{equation*}

\subsection{A slowly varying lift }
Let us go back to the issue of passing to the limit the equation \eqref{dws29}  for a general test function $\Psi$ in $ \mathcal{H}_{T}$.
The only remaining issue is to pass to the limit into the term involving   $T^n_{1}$.
We are going to use the following generalizations of Lemma 
\ref{Never0} and  Proposition \ref{propestK0}.  
We will denote $C^1_{c,\sigma} (\overline{\mathcal{F}_0} )$  the subspace of the functions in $C^1_c (\overline{\mathcal{F}_0} , \R^{2})$ which are divergence free and tangent to the boundary $\partial \mathcal{F}_0$.
\begin{Proposition}\label{propestK}
There exists a constant $M_2$ depending only on $\mathcal{F}_{0}$ such that \eqref{estK}  holds true
for any $f\in C^{1}_c (\overline{ \mathcal{F}_{0}};\R^2)$  normal to the boundary.
\end{Proposition}
Proposition \ref{propestK0}  can be proved thanks to the formula \eqref{BSexplicit}.  Actually it can also be seen as a particular case of \cite{LNX06}, Theorem $1$. 
An extension to the case of several obstacles is given in \cite{IL2S}.

Using Proposition \ref{propestK}, we can obtain the following.

\begin{Lemma}
\label{Never}
Let $ \omega $ in $\mathcal{BM} (\overline{\mathcal{F}_0} )$, diffuse (that is $\omega (\{ x\} ) =0$ for any $x \in \overline{\mathcal{F}_0}$), with vanishing total mass, and such that $u := K[ \omega  ]  \in L^2 (\mathcal{F}_0 )$.
Then \eqref{nonli} holds true for any $ \Psi \in C^1_{c,\sigma} (\overline{\mathcal{F}_0} )$. 
\end{Lemma}
\begin{proof}

Let $ \Psi \in C^1_{c,\sigma} (\overline{\mathcal{F}_0} )$.
By mollification there exists a sequence of smooth functions $\omega^{\eps}$, with vanishing total mass,  converging to $\omega$ weakly-* in $\mathcal{BM} (\overline{\mathcal{F}_0} )$ and such that  $u^{\eps}:= K[\omega^{\eps} ]$ converges strongly to $u$ in $L^2 (\mathcal{F}_0 )$.
Moreover, for any $\eps$, it follows from Lemma \ref{Never00}  that 
\begin{eqnarray}
\label{nonlieps}
\int_{\mathcal{F}_{0}} u^{\eps} \cdot ( u^{\eps} \cdot \nabla \Psi ) dx 
= - \frac{1}{2} \iint_{\mathcal{F}_{0} \times \mathcal{F}_{0} } H_{\Psi^{\perp}}(x,y) 
 \ d \omega^{\eps} (x) d \omega^{\eps} (y)  .
\end{eqnarray}
As $\eps \rightarrow 0$, the left-hand side \eqref{nonlieps} converges to the one of \eqref{nonli}.
On the other hand, we use the following lemma, borrowed from \cite{Gege}, with
$X = {\overline{\mathcal{F}_0}} \times {\overline{\mathcal{F}_0}}$, $\mu_{\eps} = \omega^{\eps} \otimes \omega^{\eps}$, $f = H_{\Psi^{\perp}}$, $F= \{(x,x) / \  x \in {\overline{\mathcal{F}_0}}  \}$,
to pass to the limit the right hand side.
\begin{Lemma}
\label{gerard}
Let $X$ be a  locally compact metric space.
Let $(\mu_{\eps})_{\eps}$ be a sequence in $\mathcal{BM} (X)$  converging to $\mu$ weakly-* in $\mathcal{BM} (X)$ and  $(\nu_{\eps})_{\eps}$ be a sequence in $\mathcal{BM}_{+} (X)$  converging to $\nu$ weakly-* in $\mathcal{BM} (X)$, with, for any $\eps$, $| \mu_{\eps} | \leq \nu_{\eps} $. 
Let $F$ be a closed subset of $X$ with $\nu (F)=0$. 
Let $f$ be a Borel bounded function in $X$ tending to $0$ at infinity, continuous on $X \setminus F$. 
Then $\int_{X} fd\mu_{\eps}  \rightarrow  \int_{X} fd\mu$.
\end{Lemma}
A proof of Lemma \ref{gerard} is provided as an appendix  for sake of completeness.
\end{proof}

	Yet the test function $\Psi$ in  $T^n_{1}$ is not normal to the boundary so that we still cannot apply  Lemma \ref{Never}. 
	The following Lemma, which is somehow reminiscent of the fake layer constructed in \cite{Sueur},   allows to correct this 
	with an arbitrarily small collateral damage.
\begin{Lemma}
\label{dualkato}
Let $\Psi \in \mathcal{H}$. Then there exists $(\tilde{\Psi}^{\eps})_{0<\eps \leq 1}$  some smooth compactly supported divergence free vector fields on $ \overline{\mathcal{F}_{0}}$ such that 
$\tilde{\Psi}^{\eps} = \ell_{\Psi}$  on $\partial {\mathcal{S}_{0}}$ and such that $\| \nabla \tilde{\Psi}^{\eps}  \|_{L^{\infty} (\mathcal{F}_{0})} \rightarrow 0$ when $\eps \rightarrow 0^{+}$.  
\end{Lemma}
\begin{proof}
Let  $\xi$ be a smooth cut-off function from $[0,+\infty)$ to $[0,1]$ with $\xi(0)=1$, $\xi'(0)=0$ and $\xi(r)=0$ for $r\geq 1$.  Then define for $x \in \overline{\mathcal{F}_{0}}$ and $0<\eps \leq 1$, 
\begin{eqnarray}
\label{2identity}
- \tilde{\Psi}^{\eps} (t,x) := \nabla^{\perp} \Big( \xi (\eps (| x | -1))  \, \ \ell_{\Psi}^{\perp} \cdot  x ) \Big) 
=  \xi (\eps (| x | -1))  \,  \ell_{\Psi}+  
 \Sigma^{\eps} ( \eps x)   .
\end{eqnarray}
where we have denoted, for $X \in \R^{2}$ with $| X | \geq \eps$,
\begin{eqnarray*}
 \Sigma^{\eps} ( X)  := \ell_{\Psi}^{\perp} \cdot  X \,  \xi' ( | X | - \eps )  \frac{1}{| X |} X^{\perp} .
\end{eqnarray*}
It is not difficult to see that $ \Sigma^{\eps}$ and $\tilde{\Psi}^{\eps}$ are smooth and compactly supported, and that $(\|   \Sigma^{\eps} ( \cdot) \|_{Lip(\mathcal{F}_{0})})_{0<\eps \leq 1}$ is bounded. 
Now that $\tilde{\Psi}^{\eps}$ is divergence free follows from the first identity in \eqref{2identity}. 
Let us now use the second one. First it shows that for $x $ in $\partial {\mathcal{F}_{0}}$, that is for $| x | = 1$, $\tilde{\Psi}^{\eps} (x) = \ell_{\Psi} (x) $.
Finally, we infer from the chain rule that  $\| \nabla \tilde{\Psi}^{\eps}  \|_{L^{\infty} (\mathcal{F}_{0})} \rightarrow 0$ when $\eps \rightarrow 0^{+}$.  
\end{proof}
 \subsection{Non-concentration of the vorticity }
 \label{Passing}
Let $\Psi$ be in $\mathcal{H}_{T}$. Lemma \ref{Never0} provides a family $(\tilde{\Psi}^{\eps})_{0<\eps \leq 1}$, the time $t$ being here a harmless parameter.
Let us also introduce 
$$ \check{\Psi}^{\eps} := {\Psi} - \tilde{\Psi}^{\eps} ,$$
which is in $C^1([0,T]; \mathcal{H} )$ and satisfies $\check{\Psi}^{\eps} \cdot \mathbf{n} = 0$ on $\partial {\mathcal{S}_{0}}$. 
We split $T^n_{1}$ into 
$T^n_{1} = \check{T}^{n,\eps}_{1} +  \tilde{T}^{n,\eps}_{1} $ with 
\begin{eqnarray*}
 \check{T}^{n,\eps}_{1} := \int_0^T  \int_{\mathcal{F}_0} u^n \cdot \left(\left(u^n \right)\cdot\nabla\right) \check{\Psi}^{\eps}  \ dx\ dt
 \text{ and } \tilde{T}^{n,\eps}_{1}:=  \int_0^T \int_{\mathcal{F}_0} u^n \cdot \left(\left(u^n \right)\cdot\nabla\right) \tilde{\Psi}^{\eps} \ dx \ dt .
\end{eqnarray*}
We are going to prove that ${T}^{n}_{1}$ converges to 
$T_{1} = \check{T}^{\eps}_{1} +  \tilde{T}^{\eps}_{1} $ with 
\begin{eqnarray*}
 \check{T}^{\eps}_{1} := \int_0^T \int_{\mathcal{F}_0} u \cdot \left(u\cdot\nabla\right) \check{\Psi}^{\eps}  \ dx\ dt , \quad  \tilde{T}^{\eps}_{1}:= \int_0^T \int_{\mathcal{F}_0} u \cdot \left(u\cdot\nabla\right) \tilde{\Psi}^{\eps} \ dx\ dt .
\end{eqnarray*}

Thanks to \eqref{conservvelo} and Lemma \ref{dualkato}, 
 $ \limsup_n |  \tilde{T}^{n,\eps}_{1}  | +  |  \tilde{T}^{\eps}_{1} | \rightarrow 0 $  when $\eps  \rightarrow 0^+ $, 
 so that in order to achieve the proof of Theorem  \ref{DelortBody}  it is sufficient to  prove that for $\eps >0$, $\check{T}^{n,\eps}_{1} \rightarrow \check{T}^{\eps}_{1}$ when $n  \rightarrow \infty$.

Actually we are going to first prove that for $\eps >0$, when $n  \rightarrow \infty$,
\begin{eqnarray}
\label{identity2}
 \int_0^T \iint_{\mathcal{F}_{0} \times \mathcal{F}_{0} }  H_{(\check{\Psi}^{\eps})^{\perp}}  (x,y) 
 \ \omega^n (x) \omega^n (y)  dx dy \ dt  \rightarrow 
 \int_0^T  \iint_{\mathcal{F}_{0} \times \mathcal{F}_{0} } H_{(\check{\Psi}^{\eps})^{\perp}}  (x,y) 
 \ \omega (x) \omega (y)  dx dy \ dt ,
\end{eqnarray}
and that $\omega (t,\cdot)$ is diffuse for almost every time $t \in (0,T)$.
Then we will apply Lemma \ref{Never} to  $f = (\check{\Psi}^{\eps})^{\perp}$ to get 
\begin{eqnarray*}\label{identity}
 \check{T}^{n,\eps}_{1} 
= - \frac{1}{2} \int_0^T  \iint_{\mathcal{F}_{0} \times \mathcal{F}_{0} } H_{ (\check{\Psi}^{\eps} )^{\perp}} (x,y)
 \ \omega^n (x) \omega^n (y)  dx dy \ dt.
\\
 \check{T}^{\eps}_{1}= - \frac{1}{2} \int_0^T  \iint_{\mathcal{F}_{0} \times \mathcal{F}_{0} } H_{ (\check{\Psi}^{\eps} )^{\perp}} (x,y)
 \ \omega (x) \omega (y)  dx dy \ dt ,
 \end{eqnarray*}
Therefore in order to achieve the proof of Theorem  \ref{DelortBody} it is sufficient to prove  the following result, which is inspired by the analysis  in  \cite{LNX06}.
\begin{Lemma}
If for any compact $K \subset \overline{\mathcal{F}_{0}}$ there exists $C>0$ such that for any $0<\delta<1$, for any $n$, 
\begin{eqnarray}
\label{identity2y}
\int_{0}^{T}  \sup_{ x \in K}  \int_{B(x, \delta) \cap  \mathcal{F}_{0}}  |\omega^{n} (t,y) | dy dt \leq C |\log  \delta|^{-1/2} ,
\end{eqnarray}
then \eqref{identity2} holds true.
\end{Lemma}
\begin{proof}
Let $\beta$ be a smooth cut-off function in $C^\infty_c (\R)$ such that $\beta (x) = 1$ for $x \leqslant 1$ and $\beta (x) = 0$ for $x \geqslant 2$.
Then define for $\delta > 0$ the function  $\beta_\delta (x) := \beta (x/\delta)$.  We split $ \check{T}^{n,\eps}_{1}$ into 
\begin{eqnarray*}
\check{T}^{n,\eps}_{1} = I^{n,\eps}_\delta + J^{n,\eps}_\delta \text{ when } I^{n,\eps}_\delta :=   - \frac{1}{2} \int_0^T  \iint_{\mathcal{F}_{0} \times \mathcal{F}_{0} } (1- \beta_\delta (| x-y | )  )H_{ (\check{\Psi}^{\eps} )^{\perp}} (x,y)   \ \omega^n (x) \omega^n (y)  dx dy \ dt   ,
 \\ J^{n,\eps}_\delta :=  - \frac{1}{2} \int_0^T  \iint_{\mathcal{F}_{0} \times \mathcal{F}_{0} }  \beta_\delta (| x-y | ) H_{ (\check{\Psi}^{\eps} )^{\perp}} (x,y)
 \ \omega^n (x) \omega^n (y)  dx dy \ dt .
\end{eqnarray*}
Let us start with $I^{n,\eps}_\delta$. We are going to prove that it converges, as $n  \rightarrow \infty$, to 
\begin{eqnarray*}
I^{\eps}_\delta :=   - \frac{1}{2} \int_0^T  \iint_{\mathcal{F}_{0} \times \mathcal{F}_{0} } (1- \beta_\delta (| x-y | )  )H_{ (\check{\Psi}^{\eps} )^{\perp}} (x,y) \ \omega (x) \omega (y)  dx dy \ dt .
\end{eqnarray*}
 For $k\geq 1$, let $\Sigma_{\frac{1}{k}} := \{ x \in \mathcal{F}_{0} / \ \dist (x, \mathcal{S}_{0} ) < 1/k \}$ and 
 $\chi_k\in C_c^\infty(\mathcal{F}_{0};[0,1])$ satisfying $\chi_k (x) = 1$ in $\mathcal{F}_{0} \setminus \Sigma_{\frac{2}{k}}$ and $\chi_k (x) = 0$ in $\Sigma_{\frac{1}{k}}$.
We decompose, for $k\geqslant 1$,  $I^{n,\eps}_\delta  - I^{\eps}_\delta =  D_1^{n,\eps,k} +  D_2^{n,\eps,k} +  D_3^{\eps,k}$  where
\begin{eqnarray*}
 D_1^{n,\eps,k} &:=&  - \frac{1}{2} \int_0^T  \iint_{\mathcal{F}_{0} \times \mathcal{F}_{0} } f_{k,\eps,\delta} (x,y) \Big( \omega^n (x) \omega^n (y) - \omega (x) \omega (y) \Big) dx dy \ dt  
 \\ D_2^{n,\eps,k} &:=&    - \frac{1}{2} \int_0^T  \iint_{\mathcal{F}_{0} \times \mathcal{F}_{0} } g_{k,\eps,\delta} (x,y)  \ \omega^n (x) \omega^n (y)  dx dy \ dt ,
 \\ D_3^{\eps,k} &:=&     \frac{1}{2} \int_0^T  \iint_{\mathcal{F}_{0} \times \mathcal{F}_{0} } g_{k,\eps,\delta} (x,y)  \ \omega (x) \omega (y)  dx dy \ dt ,
\end{eqnarray*}
with 
\begin{eqnarray*}
f_{k,\eps,\delta} (x,y) &:=& \chi_k (x) \chi_k (y)  (1- \beta_\delta (| x-y | )  )H_{ (\check{\Psi}^{\eps} )^{\perp}} (x,y)  ,
\\ g_{k,\eps,\delta} (x,y) &:=& (1 - \chi_k (x) \chi_k (y) ) (1- \beta_\delta (| x-y | )  )H_{ (\check{\Psi}^{\eps} )^{\perp}} (x,y)   .
\end{eqnarray*}
Thanks to  Lemma \ref{tensor} we get that $  (\omega^{n_{k}} \otimes \omega^{n_{k}})_{k}$  converges to $\omega \otimes  \omega $ in $C( [0,T ] ; \mathcal{BM}(\mathcal{F}_{0} \times  \mathcal{F}_{0}) - w^* )$. Since $f_{k,\eps,\delta} \in C_{0} (\mathcal{F}_{0} \times  \mathcal{F}_{0})$ we obtain that $D_1^{n,\eps,k} $ converges, as $n  \rightarrow \infty$, to $0$.

Now, using Proposition \ref{propestK} and that $\supp  (1 - \chi_k (x) \chi_k (y) )  \subset \Big( \mathcal{F}_{0} \times  \Sigma_{\frac{2}{k}} \Big) \cup  \Big( \Sigma_{\frac{2}{k}} \times \mathcal{F}_{0}  \Big) $ we obtain 
\begin{eqnarray*}
| D_3^{\eps,k} | \leq C  |\omega | ( \Sigma_{\frac{2}{k}}  ) |\omega | (\mathcal{F}_{0})  \text{ and } | D_2^{n,\eps,k} | \leq C |\omega | (\mathcal{F}_{0})  \sup_{n} \int_{0}^{T}  \sup_{ x \in K}  \int_{\Sigma_{\frac{2}{k}} }  |\omega^{n} (t,y) | dy dt    ,
 \end{eqnarray*}
which both converge to $0$ when $k  \rightarrow \infty$, thanks to \eqref{identity2y}.

Using again Proposition \ref{propestK}, we obtain  that, for $\eps >0$,  $ \sup_{n}J^{n,\eps}_\delta $ and 
\begin{eqnarray*}
J^{\eps}_\delta :=  - \frac{1}{2} \int_0^T  \iint_{\mathcal{F}_{0} \times \mathcal{F}_{0} }  \beta_\delta (| x-y | ) H_{ (\check{\Psi}^{\eps} )^{\perp}} (x,y)
 \ \omega (x) \omega (y)  dx dy \ dt 
\end{eqnarray*}
converges to $0$ when $\delta \rightarrow 0$. This entails \eqref{identity2}.
\end{proof}
\begin{Remark}
We did not succeed to prove that $(\omega^{n_{k}})_{k}$ of $(\omega^n)_{n }$  converges to $\omega := \curl {v}$ in $C( [0,T ] ; \mathcal{BM}( \overline{\mathcal{F}_{0}}) - w^* )$, so that above we have adapted Lemma \ref{gerard}  rather than applied it.
\end{Remark}
 Let us now explain how to  obtain \eqref{identity2y}. We will here also follow closely  \cite{LNX06}.
\begin{Lemma}
Let $ \phi$ be a smooth function on 
$\mathcal{F}_{0,+} $ with bounded derivatives up to second order. Then there exists $C>0$ which depends only on  $ \phi$, on $\|   \overline{v}^n_0 \|_{\rho}$
and on $ \|  \omega^n_{0} \|_{L^{1} (\mathcal{F}_{0})}$
such that 
\begin{eqnarray*}
\frac{1}{2} \int_{0}^{T}  \int_{\Gamma_+} | u^{n}  \cdot  \mathbf{n} |^{2}  \nabla \phi   \cdot \mathbf{n}^{\perp} ds \leq C  .
\end{eqnarray*}
\end{Lemma}
\begin{proof}
Using Eq. \eqref{eqvorty} and an integration by parts, we have 
\begin{eqnarray*}
\partial_t  \int_{\mathcal{F}_{0,+}}  \phi \omega^{n}  =  \int_{\mathcal{F}_{0,+}}  \phi  \partial_t  \omega^{n} 
= - \int_{\mathcal{F}_{0,+}}  \phi  (v^{n} - \ell^{n} ) \cdot \nabla \omega^{n} ,
=  \int_{\mathcal{F}_{0,+}}   \nabla \phi  \cdot (v^{n}-\ell^{n} )\omega^{n}  . 
\end{eqnarray*}
Using now the decomposition \eqref{vdecompo} we get
\begin{eqnarray}
\label{voili}
\partial_t  \int_{\mathcal{F}_{0,+}}  \phi \omega^{n}  = I_{1}^n + I_{2}^n ,
\end{eqnarray}
where 
\begin{eqnarray*}
I_{1}^n :=     \int_{\mathcal{F}_{0,+}}   \nabla \phi  \cdot   u^{n} \omega^{n}   \quad I_{2}^n := \int_{\mathcal{F}_{0,+}}   \nabla \phi  \cdot   (  \ell^{n}_1 \nabla \Phi_1 -   \ell^{n})\omega^{n}    .
\end{eqnarray*}
Using Lemma \ref{Never00} with $\Psi  :=  \nabla^\perp \phi$,  we get 
\begin{eqnarray*}
I_{1}^n = \frac{1}{2}  \int_{\Gamma_+} | u  \cdot  \mathbf{n}^{\perp} |^{2}  \nabla \phi  \cdot \mathbf{n}^{\perp} ds  
 - \int_{\mathcal{F}_{0,+}}  u^{n}  \cdot  (  u^{n}  \cdot      \nabla  \Psi ) .
\end{eqnarray*}
We now integrate in time \eqref{voili} to get 
\begin{eqnarray*}
 \frac{1}{2} \int_{0}^{T} \int_{\Gamma_+} | u^{n}  \cdot  \mathbf{n}^{\perp} |^{2}  \nabla \phi  \cdot \mathbf{n}^{\perp} ds &\leq& 
-  \int_{0}^{T} I_{2}^n
  +  \int_{\mathcal{F}_{0,+}}  \phi \omega^{n} (T,\cdot ) -  \int_{\mathcal{F}_{0,+}}  \phi \omega^{n}_{0}
 + \int_{0}^{T}  \int_{\mathcal{F}_{0,+}}  u^{n}  \cdot  (  u^{n}  \cdot      \nabla  \Psi ) .
 \end{eqnarray*}
It remains to use trivial bounds and 
\eqref{conservvelo}-\eqref{conservvorty} to conclude.
\end{proof}
Using a smooth perturbation of $\arctan (x_1 )$ instead of $\phi$ yields that for any compact $K \subset \Gamma_+$,  there exists $C>0$  such that for any $n$, 
\begin{eqnarray}
\label{ogt}
\int_{0}^{T}  \int_{K} | u^{n}  \cdot  \mathbf{n} |^{2} ds \leq C  .
\end{eqnarray}
%\\
Then one infers \eqref{identity2y}  from \eqref{conservvelo} and \eqref{ogt} following exactly the proof of 
  Lemma $2$. in \cite{LNX06}.
  This achieves the proof of Theorem \ref{DelortBody}.

 \section*{Appendix}
 
\subsection*{Proof of  \eqref{craz}}
Let us denote by
%
%\begin{eqnarray*}
$L:=- \Psi \cdot \nabla (  \frac{1}{2} | u |^{2} ) + u \cdot ( u \cdot \nabla \Psi )$ and  by $R :=  u^\perp \cdot \nabla (\Psi^\perp \cdot u )  $.
%\end{eqnarray*}
%
We extend $L$ and $R$ into 
\begin{eqnarray*}
L =     \sum_{i=1}^8 L_{i} 
= - \Psi_{1} u_{1} \partial_{1} u_{1}  - \Psi_{1} u_{2} \partial_{1} u_{2}  - \Psi_{2} ( \partial_{2}  u_{2} )  u_{2} -  \Psi_{2} ( \partial_{2}  u_{1} )  u_{1}
+ u_{1}^{2}  \partial_{1}  \Psi_{1} + u_{1} u_{2}  \partial_{1}    \Psi_{2}  + u_{1} u_{2}  \partial_{2}    \Psi_{1}  + u_{2}^{2}  \partial_{2}  \Psi_{2} ,
\\ R  =    \sum_{i=1}^8  R_{i} 
=  u_{2} (\partial_{1} \Psi_{2})  u_{1} + u_{2}  \Psi_{2}  \partial_{1} u_{1} - u_{2}^{2}  \partial_{1}  \Psi_{1} - u_{2}   \Psi_{1} \partial_{1} u_{2} 
-  u_{1}^{2} \partial_{2}  \Psi_{2} - u_{1}  \Psi_{2}  \partial_{2} u_{1} +  \Psi_{1} u_{2}  \partial_{2}  u_{2} +  u_{1} u_{2}  \partial_{2}    \Psi_{1}  ,
\end{eqnarray*}
and observe that  $L_{1} = R_{7}$,  $L_{2} = R_{4}$,  $L_{3} = R_{2}$,  $L_{4} = R_{6}$,  $L_{5} = R_{5}$, $L_{6} = R_{1}$, $L_{7} = R_{8}$, $L_{8} = R_{3}$, where we use that $u$ is divergence free for the first and third equalities, and that $\Psi$ is divergence free for the fifth and last equalities.

 \subsection*{Proof of Lemma
\ref{weakc}}
 We follow  the strategy of Appendix C. of \cite{lions}.
Let $B(0,R)$ be a ball of $X'$ containing all the values  $f_n (t)$ for all $t\in[0,T ]$, for all $n \in \N$. 
Since $X$ is separable, this ball is a compact metric space for the $w^*$ topology.
Moreover, since $Y$ is dense in $X$, one distance on this metric space is given as follows: 
let $(\phi_k )_{k\geq 1}$ be a  sequence of $Y$ dense in  $X$, and define, for $f,g$ in  $B(0,R)$,
\begin{eqnarray*}
d(f,g) := \sum_{k\geq 1} \frac{1}{2^k} \frac{| <f-g , \phi_k>_{X' ,X} | }{1+ |<f-g ,\phi_k >_{X' ,X}  | } .
\end{eqnarray*}
 Let $\eps >0$ and $k$ such that $\frac{1}{2^k} < \eps$. For any $t,s \in[0,T ]$, for all $n \in \N$,
\begin{eqnarray*}
d(f_{n} (t),f_{n} (s)) \leq \sup_{1 \leq j \leq k} | < f_{n} (t) - f_{n} (s) , \phi_j >_{X' ,X} |  +  \eps .
\end{eqnarray*}
 But using now that  $(\partial_t f_n )_n $ is  bounded in $L^{\infty}((0,T); Y')$ we get that 
 $$\sup_{1 \leq j \leq k} | < f_{n} (t) - f_{n} (s) , \phi_j >_{X' ,X} |  \rightarrow 0 \text{ when } t-s  \rightarrow 0 .$$
 Therefore the sequence  $(f_n )_n $ is  equicontinuous in $C( [0,T ];  B(0,R) - w^* )$. Thanks to the Arzela-Ascoli theorem we deduce the desired result.
 \subsection*{Proof of Lemma \ref{gerard}}
\label{gerardsec}
Let us denote by $I_\eps := \int_{X} fd\mu_{\eps} - \int_{X} fd\mu$. 
Let $\eta >0$. We are going to prove that for $\eps$ small enough, $| I_\eps | \leq 4 \eta$.
Let $M :=  \nu  (X) +   \sup_\eps \nu_{\eps} (X) $ which is finite by the Banach-Steinhaus theorem. 
Since $f$ is assumed to be decreasing at infinity, there exists a compact subset $K$ of $X$ such that $| f | \leq \eta/M $ on $X \setminus K$. 
Let us decompose $I_\eps $ into $I_\eps = I_\eps^1 + I_\eps^2$ with 
\begin{eqnarray*}
I_\eps^1  := \int_{X \setminus K} fd\mu_{\eps} - \int_{X \setminus K} fd\mu  \text{ and } 
 I_\eps^2  :=  \int_{K} fd\mu_{\eps} - \int_{K} fd\mu. 
 \end{eqnarray*}
 First we have  $| I_\eps^1 | \leq \eta$ thanks to the previous choice of $K$. It therefore remains to prove that for $\eps$ small enough, $| I_\eps^2 | \leq 3 \eta$.

 Now, let us introduce a smooth cut-off function $\xi$ on $\R$ such that $\xi (x) = 1$ for $| x | \leq 1$ and  $\xi (x) = 0$ for  $| x | \geq 2$. 
 We denote, for $\delta >0$ and $x \in X$, $\beta_\delta (x) := \xi ( \frac{\dist(x,F)}{\delta} )$. 
 We decompose $I_\eps^2$ into $I_\eps^2 = I_{\eps,\delta}^3 + I_{\eps,\delta}^4$, where 
 \begin{eqnarray*}
  I_{\eps,\delta}^3 :=  \int_{K} \beta_\delta  fd\mu_{\eps} - \int_{K}\beta_\delta  fd\mu \text{ and } 
   I_{\eps,\delta}^4 :=  \int_{K} (1-\beta_\delta ) fd\mu_{\eps} - \int_{K}(1-\beta_\delta )   fd\mu . 
 \end{eqnarray*}
   We have  
\begin{eqnarray*}
| I_{\eps,\delta}^3 | \leq \| f  \|_\infty (  \int_{K} \beta_\delta  d\nu_{\eps} +  \int_{K} \beta_\delta  d\nu )
\leq 2 \| f  \|_\infty   \int_{K} \beta_\delta  d\nu +  \eta
\end{eqnarray*}
for $\eps$ small enough, by weak-* convergence.
Since $\nu (F)=0$ there exists $\delta >0$ such that $2 \| f  \|_\infty   \int_{K} \beta_\delta  d\nu \leq  \eta$. 

Now using for this $\delta$ that $ (1-\beta_\delta ) f$ is continuous on $X$  and that $(\mu_{\eps})_{\eps}$ is  converging to $\mu$ weakly-* in $\mathcal{BM} (X)$, we get $ | I_{\eps,\delta}^4  |\leq  \eta$ for $\eps$ small enough.

Gathering all the estimates yields the result.
\\ 
\\ {\bf Acknowledgements.} The author was partially supported by the Agence Nationale de la Recherche, Project CISIFS, 
grant ANR-09-BLAN-0213-02. 

\end{document}